\documentclass[11pt]{amsart}
\usepackage{amsmath}
\usepackage{latexsym}
\usepackage{amssymb}
\usepackage{amsfonts}
\usepackage{amscd}
\usepackage{bbm}
\usepackage{color}

\input epsf

\usepackage[all]{xy}

\newtheorem{proposition}{Proposition}[section]

\newtheorem{lemma}[proposition]{Lemma}
\newtheorem{definition}[proposition]{Definition}
\newtheorem{theorem}[proposition]{Theorem}

\newtheorem{corollary}[proposition]{Corollary}

\newtheorem{remark}[proposition]{Remark}

\newtheorem{lemma-definition}[proposition]{Lemma-Definition}

\newcounter{tmp}

\def\coh{\operatorname{coh}}

\def\lto{\longrightarrow}

\def\A{{\mathcal A}}

\def\D{{\mathcal D}}
\def\F{{\mathcal F}}
\def\E{{\mathcal E}}

\def\N{{\mathcal N}}
\def\O{{\mathcal O}}

\def\M{{\mathcal M}}

\def\ZZ{{\mathbb Z}}
\def\CC{{\mathbb C}}

\def\Z{{\mathbb Z}}

\def\ZZ{{\mathbb Z}}
\def\QQ{{\mathbb Q}}
\def\CC{{\mathbb C}}

\def\PP{{\mathbb P}}

\def\Hom{\operatorname{Hom}}
\def\End{\operatorname{End}}

\def\Pic{\operatorname{Pic}}

\def\Spec{\operatorname{Spec}}

\def\id{{\operatorname{id}}}

\def\kk{{\mathbf k}}

\def\gr{\operatorname{gr}\!}

\def\tors{\operatorname{tors}\,}
\def\GrFr{\operatorname{GrFr}\!}

\def\pr{\operatorname{pr}}

\def\prd{\boxtimes}

\def\CH{\mathcal C\mathcal M}
\def\KM{\mathcal K\mathcal M}
\newcommand \uno {{\mathbbm 1}}
\newcommand \Le {{\mathbbm L}}

\title[]{Geometric Phantom Categories}

\author[]{Sergey Gorchinskiy}
\address{ Algebra Section, Steklov Mathematical Institute RAS,
Gubkin str. 8, Moscow 119991, RUSSIA}

\email{gorchins@mi.ras.ru}

\author[]{Dmitri Orlov}

\address{ Algebraic Geometry Section, Steklov Mathematical Institute RAS,
Gubkin str. 8, Moscow 119991, RUSSIA}

\email{orlov@mi.ras.ru}

\thanks{S.G. was partially supported by RFBR grants 11-01-00145, 12-01-31506, 12-01-33024, MK-4881.2011.1, NSh grant 5139.2012.1,
by AG Laboratory HSE, RF gov. grant, ag. 11.G34.31.0023.
D.O. was partially supported by  RFBR grants 10-01-93113, 11-01-00336, 11-01-00568, NSh grant 5139.2012.1,
by AG Laboratory HSE, RF gov. grant, ag. 11.G34.31.0023.}

\date{}
\dedicatory{}


\begin{document}

\begin{abstract}
In this paper we give a construction of phantom categories, i.e. admissible triangulated subcategories
in bounded derived categories of coherent sheaves on smooth projective varieties that
have trivial Hochschild homology and trivial Grothendieck group. We also prove that these phantom categories are phantoms in a stronger sense,
namely, they have trivial $K$\!-motives
and, hence, all their higher $K$\!-groups are trivial too.
\end{abstract}
\maketitle

\vspace*{-0.2cm}
\section*{Introduction}

The main purpose of this paper is to provide a construction of a phantom triangulated category. We are interested in admissible subcategories of the bounded derived categories of coherent sheaves on smooth projective varieties. A triangulated subcategory $\A\subset\mathbf{D}^b(\coh X),$ where $X$ is smooth and projective, is called admissible if it is full and the inclusion functor has a right and a left adjoint. Such categories have many good properties. Any such admissible subcategory is also saturated, i.e. smooth and proper (see \cite{KS, ToVa} and also \cite{LS} Th.3.24). The Grothendieck group $K_0(\A)$ resp. the Hochschild homology $\mathrm{HH}_*(\A)$ of an admissible subcategory are direct summands of the Grothendieck group $K_0(X)$ resp. the Hochschild homology $\mathrm{HH}_*(X)$ of the variety $X.$ There was an opinion among experts that these invariants ``see'' an admissible subcategory in the sense that they can not be trivial. Possibly nonexistent admissible subcategories with trivial Hochschild homology and a trivial Grothendieck group were called phantoms. An admissible subcategory with trivial Hochschild homology and with a finite Grothendieck group is called a quasiphantom.

Recently a few examples of quasiphantoms were constructed as semiorthogonal complements
to exceptional collections of maximal possible length on some surfaces of general type for which $q=p_g=0$ \cite{BGS, AO, GS}.
In all these cases the Grothendieck group
of a quasiphantom is isomorphic to the torsion part of the Picard group of a corresponding surface.  In the paper
\cite{BGS} authors treated the classical Godeaux surface, in this case $K_0(\A)=\ZZ/5\ZZ.$ For Burniat surfaces considered in the paper \cite{AO}
the Grothendieck group of quasiphantoms is $(\ZZ/2\ZZ)^6.$ In the paper \cite{GS} authors studied the Beauville surface and obtained
a quasiphantom with $K_0(\A)=(\ZZ/5\ZZ)^2.$

These results allow us to hope for the existence of a phantom as a semiorthogonal complement to an exceptional collection of maximal length on a
simply connected surface of general type with $q=p_g=0$ like a Barlow surface  (recall that a simply connected surface has a trivial torsion part of the Picard group).

On the other hand, we can try to use another approach. It is also natural to consider an admissible subcategory generated by the tensor product
of two quasiphantoms $\A$ and $\A'$ for which orders of $K_0(\A)$ and $K_0(\A')$ are coprime.  For any pair of quasiphantoms $\A\subset\mathbf{D}^b(\coh S)$ and \mbox{$\A'\subset\mathbf{D}^b(\coh S')$} we can take
a full triangulated subcategory $\A\prd\A'$ of the category
$\mathbf{D}^b(\coh(S\times S'))$ that is the minimal triangulated subcategory
closed under taking direct summands and containing
all objects of the form $\pr_1^* \E\otimes \pr_2^* \F$ with $\E\in \A$ and $\F\in \A'.$
Assuming that~$S$ and $S'$ are two different smooth projective surfaces of general type over $\CC$ with $q=p_g=0$
for which orders of the torsion parts of Picard groups are coprime, we prove that in this case the admissible subcategory
$\A\prd\A'$ is a phantom (Theorem \ref{main}). We also show that the category $\A\prd\A'$ is a phantom in a strong sense.  We introduce
a notion of universal phantom (see Definition \ref{def:uphantom}) and prove that the property for an admissible subcategory
 $\N\subset \mathbf{D}^b(\coh X)$ to be a universal phantom is equivalent to the vanishing of its $K$\!-motive
 $KM(\N)$ (see Proposition \ref{prop:KM-univ}). In fact
 we show that the admissible subcategory $\A\prd\A'$ is a universal phantom (Theorem \ref{main2}).
This immediately implies that all its $K$\!-groups $K_i(\A\prd\A')$ are trivial.
All these results can be applied to the case when $S$ is a Burniat surface and $S'$ is the classical Godeaux surface (or the Beauville surface).
This gives us first examples of geometric phantom categories. Actually, we obtain first  examples of  saturated DG categories whose
$K$\!-motives (that are also called noncommutative motives) are trivial and, moreover, these examples have a geometric nature, i.e. they are admissible subcategories in bounded derived categories of coherent sheaves on smooth projective varieties.

We believe that new examples of surfaces of general type with
exceptional collections of maximal length will be found in the near future and so new examples of quasiphantoms
will be obtained. Applying Theorems \ref{main} and \ref{main2} we will be able to get
other universal phantom categories as well.

In the last section we also show that over $\CC$ the vanishing of Hochschild homology for admissible subcategories is a consequence of the vanishing of $K_0$\!-groups with rational coefficients.

When this work was done we were informed by Ludmil Katzarkov that the approach of constructing of a phantom as  a semiorthogonal complement to an exceptional collection of maximal length on a Barlow surface is realized now by Ch.~B\"ohning, H-Ch.~Graf~von Bothmer, L.~Katzarkov, and P.~Sosna  in their incoming paper \cite{BBKS}.

We would like to thank Ivan Panin for very useful discussions and comments.

\section{Semiorthogonal decompositions and phantoms}

Let $\D$ be a $\kk$\!-linear triangulated category category, where $\kk$ is a base field.
Recall some definitions and facts concerning admissible
subcategories and semiorthogonal decompositions (see \cite{BK, BO}).
Let ${\N\subset\D}$ be a full triangulated subcategory.  The {\sf
right orthogonal} to ${\N}$ is the full subcategory
${\N}^{\perp}\subset {\D}$ consisting of all objects $M$ such that
${\Hom(N, M)}=0$ for any $N\in{\N}.$ The {\sf left orthogonal}
${}^{\perp}{\N}$ is defined analogously.  The orthogonals are also
triangulated subcategories.

\begin{definition}\label{adm}
Let $I\colon\N\hookrightarrow\D$ be an embedding of a full triangulated
subcategory $\N$ in a triangulated category $\D.$ We say that ${\N}$
is {\sf right admissible} (respectively {\sf left admissible}) if
there is a right (respectively left) adjoint functor $Q\colon\D\to \N.$ The
subcategory $\N$ will be called {\sf admissible} if it is right and left
admissible.
\end{definition}
\begin{remark}\label{semad}
{\rm For the subcategory $\N$ the property of being right admissible
is equivalent to requiring that for each $X\in{\D}$ there be an
exact triangle $N\to X\to M,$ with $N\in{\N}, M\in{\N}^{\perp}.$ }
\end{remark}

Let $\N$ be a full triangulated subcategory in a triangulated category
$\D.$ If $\N$ is right (respectively left) admissible, then the
quotient category $\D/\N$ is equivalent to $\N^{\perp}$ (respectively~${}^{\perp}\N$).  Conversely, if the quotient functor $Q\colon\D\lto\D/\N$
has a left (respectively right) adjoint, then~$\D/\N$ is equivalent to
$\N^{\perp}$ (respectively~${}^{\perp}\N$).

\begin{definition}\label{sd}
A collection of admissible subcategories $({\N}_1, \dots, {\N}_n)$ in a triangulated category~${\D}$
is said to be {\sf semiorthogonal} if the condition ${\N}_j\subset {\N}^{\perp}_i$ holds
when $j<i$ for any $1\le i\le n.$
A semiorthogonal collection is said to be {\sf full} if it
generates the category~${\D},$ i.e. the minimal triangulated subcategory of $\D$ containing
all $\N_j$ coincides with the whole $\D.$
In this case we call it by  a {\sf semiorthogonal decomposition}
of the category ${\D}$ and denote this as
$$
{\D}=\left\langle{\N}_1, \dots, {\N}_n\right\rangle.
$$
\end{definition}

The existence of a semiorthogonal decomposition on a triangulated
category $\D$ clarifies the structure of $\D.$ In the best scenario,
one can hope that $\D$ has a semiorthogonal decomposition
${\D}=\left\langle{\N}_1, \dots, {\N}_n\right\rangle$ in which each
elementary constituent $\N_p$ is as simple as possible, i.e. is
equivalent to the bounded derived category of finite-dimensional
vector spaces.


\begin{definition}\label{exc}
An object $E$ of a $\kk$\!-linear triangulated category ${\D}$ is
called {\sf exceptional} if  ${\Hom}(E, E[l])=0$ when $l\ne 0,$ and
${\Hom}(E, E)=\kk.$ An {\sf  exceptional collection} in ${\D}$ is a
sequence of exceptional objects $(E_1,\dots, E_n)$ satisfying the
semiorthogonality condition ${\Hom}(E_i, E_j[l])=0$ for all $l$ when
$i>j.$
\end{definition}

If a triangulated category $\D$ has an exceptional collection
$(E_1,\dots, E_n)$ that generates the whole of $\D$ then we say that
the collection is {\sf full}.  In this case $\D$ has a
semiorthogonal decomposition with $\N_p=\langle E_p\rangle.$ Since
$E_{p}$ is exceptional, each of these categories is equivalent to the
bounded derived category of finite dimensional vector spaces.  In
this case we write $ \D=\langle E_1,\dots, E_n \rangle.$

\begin{definition}\label{strong}
An exceptional collection $(E_1,\dots, E_n)$ is called strong if, in addition,
${\Hom}(E_i, E_j[l])=0$ for all  $i$ and $j$ when $l\ne 0.$
\end{definition}

The best known example of an exceptional collection is the sequence
of invertible sheaves
$\left(\mathcal{O}_{\mathbb{P}^n},\dots,\mathcal{O}_{\mathbb{P}^n}(n)\right)$
on the projective space $\mathbb{P}^n.$ This exceptional collection is
full and strong.

If we have a semiorthogonal decomposition ${\D}=\left\langle{\N}_1, \dots, {\N}_n\right\rangle$ then the inclusion functors induce an isomorphism on Grothendieck groups
$$
K_0(\N_1)\oplus K_0(\N_2)\oplus\cdots\oplus K_0(\N_n)\cong K_0(\D).
$$

For example, if $\D$ has a full exceptional collection then the Grothendieck group
$K_0(\D)$ is a free abelian group $\ZZ^n.$

It is more convenient to consider so called enhanced triangulated categories, i.e. triangulated categories
that are  homotopy categories  of pretriangulated differential graded (DG) categories (see \cite{BK2, Ke2}).
Any geometric category like the bounded derived category of coherent sheaves
$\mathbf{D}^b(\coh X)$
has a natural enhancement. It is even shown that in many cases these categories have a unique enhancement \cite{LO}.
An enhancement of a triangulated category  $\D$ induces an enhancement for any full triangulated subcategory
$\N\subset\D.$

Using enhancement of a triangulated category $\D$ we can define $K$\!-theory spectrum $K(\D)$ and Hochschild homology $\mathrm{HH}_{*}(\D)$ of $\D$
(see \cite{Ke2}).  They also give us additive invariants (see, for example, \cite{Ke2}  5.1, Th. 5.1 c) and Th. 5.2 a) and also \cite{Ku1} Th.7.3 for Hochschild homology), i.e for any
semiorthogonal decomposition we obtain  isomorphisms
$$
\bigoplus_{i=1}^{n} \mathrm{HH}_*(\N_i)\cong \mathrm{HH}_{*}(\D)  \quad\text{and}\quad
\bigoplus_{i=1}^{n} K_*(\N_i)\cong K_*(\D)
$$

We are interested in categories that have a geometric nature, i.e. when $\D$ and $\N_i$ are full triangulated subcategories of
the bounded derived category of coherent sheaves on a smooth projective variety $X.$ In this case, as it was mentioned above,
we have natural enhancements for these categories and we work with such enhancements.

Let $X$ and $Y$ be two smooth projective varieties. Any semiorthogonal decompositions
\begin{equation}\label{semi} \mathbf{D}^b(\coh X)= \left\langle{\N}_1, \dots, {\N}_n\right\rangle, \quad
\mathbf{D}^b(\coh Y)= \left\langle{\M}_1, \dots, {\M}_m\right\rangle
\end{equation} induce a semiorthogonal decomposition of the product $\mathbf{D}^b(\coh(X\times Y)).$
Indeed, for any pair of subcategories $\N_i$ and $\M_j$ we can define
a full triangulated subcategory $\N_i\prd\M_j$ of the category
$\mathbf{D}^b(\coh(X\times Y))$ as the minimal triangulated subcategory
of $\mathbf{D}^b(\coh(X\times Y))$ closed under taking direct summands and containing
all objects of the form $\pr_1^* N\otimes \pr_2^* M$ with $N\in \N_i$ and $M\in \M_j.$

It is easy to see that the subcategories $\N_i\prd\M_j$    and $\N_k\prd\M_l$   are semiorthogonal
in the sense that $\N_i\prd\M_j\subset (\N_k\prd\M_l)^{\perp}$ when $i<k$ or $j<l.$

On the other hand, the whole category $\mathbf{D}^b(\coh(X\times Y))$ is split generated
by objects of the form $\pr_1^* A\otimes \pr_2^*B,$ where $A\in \mathbf{D}^b(\coh X)$ and
$B\in \mathbf{D}^b(\coh Y),$ i.e. the minimal triangulated subcategory
of $\mathbf{D}^b(\coh(X\times Y))$ closed under taking direct summands and containing
all objects of the form $\pr_1^* A\otimes \pr_2^* B$ coincides with the whole category $\mathbf{D}^b(\coh(X\times Y)).$
It follows from the fact that any sheaf $\F$ on the product has a resolution $P^{\cdot}$ by sheaves of such form
and this sheaf $\F$ is a direct summand of the stupid truncation $\sigma_{\ge -p}P^{\cdot}$ for sufficient large $p.$
Therefore the subcategories  $\N_i\prd\M_j$ generates the whole category $\mathbf{D}^b(\coh(X\times Y))$
and we obtain a semiorthogonal decomposition for the product $X\times Y.$ Thus we proved the following statement.
\begin{proposition} The semiorthogonal decompositions (\ref{semi}) of
$\mathbf{D}^b(\coh X)$ and $\mathbf{D}^b(\coh Y)$ give a semiorthogonal decomposition for the product
$$
\mathbf{D}^b(\coh(X\times Y))= \left\langle  \N_i\prd\M_j \right\rangle_{1\le i\le n, 1\le j \le m}.
$$
\end{proposition}
\begin{remark} {\rm As it was shown in \cite{To}, in the `world of DG categories up to quasi-equivalences' the category
$\N_i\prd\M_j$ can be considered  as a category of functors from $\N_i$ to $\M_j$ (see also \cite{Ke2}).
This means that the category $\N_i\prd\M_j$ can be defined abstractly without using the full embedding
to a bounded derived category of coherent sheaves.
}
\end{remark}

\begin{definition} An admissible triangulated subcategory
$\A$ in $\mathbf{D}^b(\coh X),$ where $X$ is a smooth projective variety
will be called  a {\sf quasiphantom} if
$\mathrm{HH}_{*}(\A)=0$ and $K_0(\A)$ is a finite abelian group.
It is called a {\sf phantom} if, in addition, $K_0(\A)=0.$
\end{definition}

\begin{definition}\label{def:uphantom}
We say that a phantom subcategory $\A\subset\mathbf{D}^b(\coh X)$ is a {\sf universal phantom} if
$\A\prd\mathbf{D}^b(\coh Y)\subset\mathbf{D}^b(\coh(X\times Y))$ is also phantom
for any smooth projective variety $Y.$
\end{definition}

As it will be shown in section \ref{sec:Kmotives} this universal property can be checked for $Y=X,$ i.e. if $\A\prd\mathbf{D}^b(\coh X)\subset
\mathbf{D}^b(\coh(X\times X))$ is a phantom then $\A$ is a universal phantom.

Recently \cite{BGS, AO, GS} different examples of quasiphantoms were constructed as semiorthogonal complements
to exceptional collections of maximal length on some surfaces of general type with $q=p_g=0$ for which Bloch's conjecture holds by \cite{IM}.

In more detail, let $S$ be a smooth projective surface of general type over $\CC$ with $q=p_g=0$ for which Bloch's conjecture for
0-cycles holds, i.e. the Chow group $CH^2(S)\cong \ZZ.$ In this case the Grothendieck group
$K_0(S)$  is  isomorphic to $\ZZ\oplus \Pic(S)\oplus \ZZ\cong \ZZ^{r+2}\oplus \Pic(S)_{\tors},$ where
$r$ is the rank of the Picard lattice $\Pic(S)/\tors.$
Since for such a surface $\Pic(S)$ is isomorphic to $H^2(S(\CC), \ZZ)$ and using Noether's formula we have
$$
r+2=b_2+2=e=c_2=12-c_1^2,
$$
where $b_2$ is the second Betti number, $e$ is the topological Euler characteristic,  and $c_1, c_2$ are the first and the second Chern classes of $S.$

Assume that the derived category
 $\mathbf{D}^b(\coh S)$ possesses an exceptional collection $(E_1,\ldots, E_{e})$ of the maximal length $e.$
In this case we obtain a semiorthogonal
decomposition
$$
\mathbf{D}^b(\coh S)=\langle \D, \A \rangle,
$$
where $\D$ is the admissible subcategory generated by $(E_1,\ldots, E_{e})$ and  $\A$ is the left orthogonal to $\D.$
We have that $K_0(\D)\cong\ZZ^{e}$ and $K_0(\A)\cong\Pic(S)_{\tors}.$

For the classical Godeaux surface $S$ that is the
$\ZZ/5\ZZ$-quotient of the Fermat quintic in~$\PP^3$ an exceptional collection of maximal length was constructed
in \cite{BGS}. In this case $e=11.$

\begin{theorem}\cite[Th. 8.2]{BGS}\label{th-god}
Let $S$ be the classical Godeaux surface. There exists a semi-
orthogonal decomposition
$$
\mathbf{D}^b(\coh S)=\langle L_1,\ldots, L_{11}, \A \rangle,
$$
where $(L_1, \ldots, L_{11})$ is an exceptional sequence of maximal length consisting
of line bundles on $S.$
\end{theorem}
The subcategory $\A$ for the classical Godeaux surface is a quasiphantom and the Grothendieck group $K_0(\A)$ is isomorphic to the cyclic group $\ZZ/5\ZZ.$

For Burniat surfaces with $c_1^2=6$  exceptional collections of maximal length 6 were constructed
in \cite{AO}. In this case we have a 4-dimensional family of such surfaces.

\begin{theorem}\cite[Th. 4.12]{AO}\label{th-bur} For any Burniat surface $S$ we have a semiorthogonal
  decomposition
$$
\mathbf{D}^b(\coh S)=\langle \D, \A \rangle,
$$
where $\D$ is an admissible subcategory generated by an exceptional collection of line bundles
$(L_1, \ldots, L_6).$ The category $\D$ is the same for all
Burniat surfaces. The category $\A$ has trivial Hochschild homology
and $K_0(\A)=(\ZZ/2\ZZ)^6.$
\end{theorem}

For Burniat surfaces we obtain a family of quasiphantoms $\A$ with $K_0(\A)=(\ZZ/2\ZZ)^6.$
In the recent paper of A. Kuznetsov \cite{Ku2} it is proved that the second Hochschild cohomology
of $\A$ coincides with the second Hochschild cohomology of a Burniat surface $S$ that is actually $H^1(S, T_S)$
and has dimension 4.

In the paper \cite{GS} the authors considered the  Beauville surface and constructed
a quasiphantom the Grothendieck group of which is isomorphic to $(\ZZ/5\ZZ)^2.$

Using two different quasiphantoms $\A$ and $\A'$ we can try to construct a phantom category taking the product
$\A\prd\A'.$ If the orders of Grothendieck groups $K_0(\A)$ and $K_0(\A')$ are coprime we can hope that the Grothendieck
group of $\A\prd\A'$ will be trivial. In the case of surfaces we prove this. The following two theorems are the main results of the paper.

\begin{theorem}\label{main}
Let $S$ and $S'$ be  smooth projective surfaces over $\CC$ with $q=p_g=0$ for which Bloch's conjecture for 0-cycles holds.
Assume that the derived categories $\mathbf{D}^b(\coh S)$ and $\mathbf{D}^b(\coh S')$ have exceptional collections of maximal
lengths $e(S)$ and $e(S'),$ respectively. Let $\A\subset\mathbf{D}^b(\coh S)$ and $\A'\subset\mathbf{D}^b(\coh S')$ be the left orthogonals
to these exceptional collections. If the orders of $\Pic(S)_{\tors}$ and $\Pic(S')_{\tors}$ are coprime, then the admissible subcategory \mbox{$\A\prd\A'\subset \mathbf{D}^b(\coh(S\times S'))$}
is a phantom category, i.e. $\mathrm{HH}_*(\A\prd\A')=0$ and $K_0(\A\prd\A')=0.$
\end{theorem}
\begin{remark}\label{rem:Hochsch}
{\rm It is evident that the category $\A\prd\A'$ is not trivial, because  any object of the form $\pr_1^* A\otimes\pr_2^* A',$ where $A\in\A$ and $A'\in\A,$ belongs to
the category $\A\prd\A'.$ It is easy to see that the Hochschild homology of $\A\prd\A'$ are trivial. Indeed, Hochschild homology of $\A$ and~$\A'$ are trivial and  there is an isomorphism $\mathrm{HH}_i(X)=\bigoplus_p H^{p+i}(X, \Omega^p_X)$ for a smooth projective variety $X.$ Thus the K\"unneth formula implies triviality
of Hochschild homology of $\A\prd\A'.$ (See also Section~\ref{sec:add} for an alternative approach.)
}
\end{remark}

Actually, we can prove a stronger result.

\begin{theorem}\label{main2} The phantom category $\A\prd\A'\subset \mathbf{D}^b(\coh(S\times S'))$ from Theorem \ref{main}
is a universal phantom and it has a trivial K-theory, i.e. $K_*(\A\prd\A')=0.$
\end{theorem}

This theorem together with Proposition \ref{prop:KM-univ} says us that the category  $\A\prd\A'$ has a trivial $K$\!-motive, i.e.
it is in the kernel of the natural map from the world of saturated DG categories to the world of K-motives (they now are
called noncommutative motives).

The main tool for the proof of Theorems \ref{main}, \ref{main2} is the Merkurjev--Suslin result on $K_2$\!-groups
and its corollaries
that allow to control torsion in Chow groups of cycles of codimension 2~\cite{MS}.

Applying Theorem \ref{main2} to recently known examples we obtain a corollary.

\begin{corollary} Let $S$ be a Burniat surface with $c_1^2=6$ and let $S'$ be the classical Godeaux surface over $\CC.$
Let $\A$ and $\A'$ be quasiphantoms from Theorems \ref{th-bur} and \ref{th-god}, respectively. Then the category
$\A\prd\A'\subset\mathbf{D}^b(\coh(S\times S'))$ is a universal phantom category and $K_*(\A\prd\A')=0.$
\end{corollary}
\begin{remark}{\rm Due to \cite{GS} instead of the classical Godeaux surface we can take a Beauville surface.
}
\end{remark}
Of course, Theorem~\ref{main2} implies Theorem~\ref{main}. However, we give two separate proofs of these theorems. In Section~\ref{sec:K} we prove Theorem~\ref{main} using results on Chow groups and $K_0$-groups. In Section~\ref{sec:Kmotives} we prove Theorem~\ref{main2} using $K$\!-motives and not using Theorem~\ref{main}.

\section{Chow motives of surfaces of general type with $q=p_g=0$}\label{sec:CH}

First let us fix notation and recollect some general facts concerning Chow motives (general references are~\cite{Manin} and~\cite{Scholl}).
Let $\kk$ be a field. By $X$ we denote an irreducible smooth projective variety over $\kk$ and by $d$ we denote the dimension of $X.$ By $CH^p(X)$ denote the Chow group of codimension $p$ cycles on $X.$ Given a codimension $p$ cycle $Z$ on $X,$ by $[Z]$ denote its class in $CH^p(X).$ Given irreducible smooth projective varieties $X,$ $Y,$ $Z$ and elements \mbox{$f\in CH^p(X\times Y)$}, $g\in CH^q(Y\times Z),$ i.e. correspondences, put
$$
g\circ f:=\pr_{13\,*}(\pr_{12}^*(f)\cdot\pr_{23}^*(g))\,,
$$
where $\pr_{ij}$ denote natural projections from $X\times Y\times Z.$

By $\CH(\kk)$ denote the category of Chow motives over $\kk$ with integral coefficients. Objects in $\CH(\kk)$ are given by triples $M(X,\pi,n),$ where $\pi\in CH^d(X\times X)$ satisfies $\pi\circ\pi=\pi,$ i.e. $\pi$ is a projector, and $n$ is an integer. For $n=0,$ we usually omit $n$ in the latter triple. Morphisms between Chow motives are defined by the formula
$$
\Hom(M(X,\pi,m),M(Y,\rho,n)):=\rho\circ CH^{d+n-m}(X\times Y)\circ\pi
$$
and composition of morphisms is given by composition of correspondences. The Lefschetz motive is defined by the formula $\Le:=(\Spec(\kk),[\Delta],-1)$ and, similarly, $\Le^n:=(\Spec(\kk),[\Delta],-n)$ for~\mbox{$n\in \Z$}.

We have a Chow motive $M(X):=M(X,[\Delta_X]),$ where $\Delta_X\subset X\times X$ is the diagonal. This defines a contravariant functor $X\mapsto M(X)$ from the category of smooth projective varieties over~$\kk$ to $\CH(\kk).$ One has a canonical isomorphism $M(\PP^1)\cong\uno\oplus\Le.$

There is a symmetric tensor structure on $\CH(\kk)$ defined by the formula
$$
M(X,\pi,m)\otimes M(Y,\rho,n):=M(X\times Y,\pi\times\rho,m+n)\,.
$$
The unit object $\uno$ is $M(\Spec(\kk)).$ Moreover, $\CH(\kk)$ is rigid with $M(X,\pi,n)^{\vee}$ being isomorphic to $M(X,\pi^{\mathrm t},d-n),$ where $\pi^{\mathrm t}$ denotes the image of $\pi$ under the natural symmetry on~\mbox{$X\times X$}. In particular, for any natural $n,$ we have $\Le^n\cong \Le^{\otimes n}$ and
$\Le^{-n}\cong (\Le^{\vee})^{\otimes n}.$

Chow groups of Chow motives are defined by the formula $CH^p(X,\pi,n):=\pi(CH^{p+n}(X)),$ where $\pi(-):=\mathrm{pr}_{2*}(\mathrm{pr}_1^*(-)\cdot \pi)$ and $\mathrm{pr}_i\colon X\times X\to X$ are natural projections. In particular, $CH^1(\Le)\cong\ZZ$ and $CH^p(\Le)=0$ for $p\ne 1.$

Any element $\alpha\in CH^p(X)$ defines morphisms
$$
\alpha_*\colon\Le^{p}\to M(X),\quad \alpha^*\colon M(X)\to\Le^{(d-p)}\,,
$$
because $\Spec(\kk)\times X=X\times \Spec(\kk)=X.$ Given another element $\beta\in CH^{d-p}(X),$ we have the equalities
\begin{equation}\label{eq:compos}
\beta^*\circ \alpha_*=(\alpha.\beta)\cdot\id_{\Le^{p}}\,,\quad \alpha_*\circ\beta^*=\beta\times \alpha\,,
\end{equation}
where $(\alpha.\beta)$ denotes the intersection pairing.

Under the isomorphism $M(X)^{\vee}\cong M(X)\otimes\Le^{-d}$ the evaluation morphism $M(X)\otimes M(X)^{\vee}\to\uno$ corresponds to a morphism $M(X)\otimes (M(X)\otimes\Le^{-d})\to\uno.$ The latter morphism is induced by the morphism $[\Delta_X]^*\colon M(X\times X)\to\Le^{d}.$ It follows that, after taking Chow groups, the evaluation morphism corresponds to the intersection pairing.

We say that a Chow motive is of Lefschetz type if it is a direct sum of finitely many motives $\Le^p$ for some (possibly, different) numbers $p.$

For a commutative ring $R,$ one defines similarly the category of Chow motives with \mbox{$R$-coefficients} $\CH(\kk)_R$ based on Chow groups with \mbox{$R$-coefficients} $CH^p(X)_R:=CH^p(X)\otimes_{\ZZ}R.$ We denote Chow motives with $R$-coefficients by $M(X,\pi,n)_R.$ For simplicity, we use the same notations $\uno$ and $\Le$ for the corresponding objects in $\CH(\kk)_R.$ All what was said above about the category $\CH(\kk)$ remains valid for the category $\CH(\kk)_R.$

Given an integer $N,$ we use a standard notation $\ZZ[\frac{1}{N}]$ for the localization of the ring $\ZZ$ over all powers of $N.$ For simplicity, when $R=\ZZ[\frac{1}{N}]$ we use the index $1/N$ instead of $\ZZ[\frac{1}{N}],$ i.e. we use the notation $\CH(\kk)_{1/N},$ $CH^i(X)_{1/N},$ and $M(X,\pi,n)_{1/N}.$

\begin{lemma}\label{lemma-inters}
If the Chow motive $M(X)_R$ of an irreducible smooth projective variety $X$ is of Lefschetz type, then $CH^{\bullet}(X)_R$ is a free $R$-module of finite rank and the intersection pairing on $CH^{\bullet}(X)_R$ is unimodular.
\end{lemma}
\begin{proof}
The only non-trivial part is to show that the intersection pairing is unimodular.
Let $\CH_{\Le}(\kk)_R$ be the full subcategory in $\CH(\kk)_R$ formed by Chow motives of Lefschetz type. Note that $\CH_{\Le}(\kk)_R$ is a rigid symmetric tensor category with the same dual objects as in $\CH(\kk)_R.$

By $\GrFr_R$ denote the rigid symmetric tensor category of graded free finite rank $R$-modules (the symmetric structure is defined in the simplest way without the Koszul sign rule). Then we have an equivalence of symmetric tensor categories
$$
CH^{\bullet}\colon\CH_{\Le}(\kk)_R\stackrel{\sim}\longrightarrow \GrFr_R\,.
$$
Therefore the intersection pairing coincides with the evaluation map for $CH^{\bullet}(X)=CH^{\bullet}(M(X)).$ Thus the intersection pairing is unimodular.
\end{proof}

The following two propositions are the main ingredients for the proofs of  Theorems~\ref{main} and  \ref{main2}.

\begin{proposition}\label{prop-Chowdecomp}
Let $S$ be a smooth projective surface over $\CC$ with $q=p_g=0$ for which Bloch's conjecture for 0-cycles holds,
i.e. $CH^2(S)\cong\ZZ.$ Let $r$ be the rank of the Picard lattice $\Pic(S)/\tors.$ Then there is an isomorphism in the category of Chow motives $\CH(\CC)$:
$$
M(S)\cong \uno\oplus\Le^{\oplus r}\oplus\Le^{2}\oplus M
$$
such that \mbox{$CH^1(M)=\Pic(S)_{\tors}$} and $CH^p(M)=0$ for $p\ne 1.$
\end{proposition}
\begin{proof}
Take a point $s\in S$ and consider the projectors $\pi_0:=[s\times S]$ and $\pi_4:=[S\times s]$ in $CH^2(S\times S).$ It is well-known that $(S,\pi_0)\cong \uno$ and $(S,\pi_4)\cong \Le^{2}.$

Since $H^i(S,\O_S)=0$ for $i=1,2,$ it follows from the exponential exact sequence that $\Pic(S)\cong H^2(S(\CC),\Z),$ where $S(\CC)$ denotes the topological space of complex points on $S.$ Hence Poincar\'e duality for singular cohomology asserts that the intersection product on $\Pic(S)/\tors$ is unimodular. Let $\{D_i\},$ $1\leqslant i\leqslant r,$ be a collection of divisors on $S$ whose classes give a basis in $\Pic(S)/{\tors}.$ Further, let $\{E_i\},$ $1\leqslant i\leqslant r,$ be a collection of divisors on $S$ whose classes give the dual basis in $\Pic(S)/{\tors}.$ It follows from equation~\eqref{eq:compos} that $\pi_2:=\sum\limits_{i=1}^r [D_i\times E_i]$ is a projector in $CH^2(S\times S)$ and that there is an isomorphism of Chow motives $(S,\pi_2)\cong\Le^{\oplus r}.$

By codimension reason, the projectors $\pi_0,$ $\pi_2,$ $\pi_4$ are orthogonal, i.e. $\pi_i\circ\pi_j=0$ for $i\ne j.$ Therefore we obtain a projector $\pi:=[\Delta_S]-\pi_0+\pi_2+\pi_4$ and a decomposition
$$
M(S)\cong \uno\oplus\Le^{\oplus r}\oplus\Le^{2}\oplus M\,,
$$
where $M:=M(S,\pi).$ It follows that \mbox{$CH^1(M)=\Pic(S)_{\tors}$} and $CH^p(M)=0$ for $p\ne 1.$
\end{proof}

\begin{proposition}\label{prop-Chow}
Let $S,$ $r,$ and $M$ be as in Proposition~\ref{prop-Chowdecomp}. Let $N$ be a non-zero integer such that $N\cdot\Pic(S)_{\tors}=0.$ Then $M_{1/N}=0$ and $M(S)_{1/N}$ is of Lefschetz type, namely, there is an isomorphism in the category $\CH(\CC)_{1/N}$:
$$
\mbox{$
M(S)_{1/N}\cong \uno\oplus \Le^{\oplus r}\oplus \Le^{2}\,. $}
$$
\end{proposition}
\begin{proof}
Clearly, it is enough to show that $M_{1/N}=0.$ Let $M=M(S,\pi),$ $\pi\in CH^2(S\times S),$ be the motive from the decomposition given by Proposition~\ref{prop-Chowdecomp}. We claim that the motive $M_{\QQ}$ vanishes in $\CH(\CC)_{\QQ}.$ There are several possible ways to do this (for example, see Proposition 14.2.3 and Corollary 14.4.9 in~\cite{KMP}). Let us explain one of them. Lemma~1 from~\cite{GG} says that a Chow motive in $\CH(\CC)_{\QQ}$ with trivial Chow groups is trivial (one applies this lemma with $\Omega=\CC$ and $k$ being a minimal field of definition of a given Chow motive). By Proposition~\ref{prop-Chowdecomp}, all groups $CH^p(M)$ are torsion. Thus all Chow groups $CH^p(M_{\QQ})=CH^p(M)_{\QQ}$ vanish and by the above we have $M_{\QQ}=0.$ Consequently, the element $\pi\in CH^2(S\times S)$ is torsion.

Now let us prove that $N\cdot\pi=0$ in $CH^2(S\times S).$ For short, put $T:=S\times S.$ By a result of Merkurjev and Suslin, see Corollary 18.3 in \cite{MS}, the group $CH^2(T)_{\tors}$ is a subquotient (actually a subgroup) of the group $H^3_{\acute e t}(T,\QQ/\ZZ(2)).$ Choosing a compatible system of roots of unity, we get an isomorphism $H^3_{\acute e t}(T,\QQ/\ZZ(2))\cong H^3(T(\CC),\QQ/\ZZ).$

Let us describe the latter group explicitly. The exact sequence of constant sheaves on~$T(\CC)$
$$
0\to\ZZ\to\QQ\to\QQ/\ZZ\to 0
$$
leads to the exact sequence of abelian groups
\begin{equation}\label{eq:QZ}
0\to H^3(T(\CC),\QQ)\,/\,H^3(T(\CC),\ZZ)\to H^3(T(\CC),\QQ/\ZZ)\to H^4(T(\CC),\ZZ)_{\tors}\to 0\,.
\end{equation}
Poincar\'e duality for $S(\CC)$ implies that $H^i(S(\CC),\QQ)=0$ for odd $i$ and $N\cdot H^{\bullet}(S(\CC),\ZZ)_{\tors}=0.$ K\"unneth formula shows that the same is true for $T=S\times S$: $H^i(T(\CC),\QQ)=0$ for odd $i$ and $N\cdot H^{\bullet}(T(\CC),\ZZ)_{\tors}=0.$ Therefore, by equation~\eqref{eq:QZ}, we have
$N\cdot H^3(T(\CC),\QQ/\ZZ)=0.$

We conclude that $N\cdot CH^2(S\times S)_{\tors}=0.$ In particular,
$N\cdot \pi=0,$ whence $M_{1/N}=0.$ This finishes the proof.
\end{proof}

\begin{corollary}\label{corol-Chow}
Assume that surfaces $S$ and $S'$ satisfy all conditions from Proposition~\ref{prop-Chowdecomp}. Let~$M$ and $M'$ be the Chow motives from Proposition~\ref{prop-Chowdecomp} and let $N$ and $N'$ be the integers from Proposition~\ref{prop-Chow} corresponding to $S$ and $S',$ respectively. Suppose that $N$ and $N'$ are coprime. Then the following is true:
\begin{itemize}
\item[(i)]
the tensor product vanishes $M\otimes M'=0$ in the category $\CH(\kk)$;
\item[(ii)]
the external product map is an isomorphism:
$$
CH^{\bullet}(S)\otimes_{\ZZ}CH^{\bullet}(S')\stackrel{\sim}\longrightarrow CH^{\bullet}(S\times S')\,.
$$
\end{itemize}
\end{corollary}
\begin{proof}
By Proposition~\ref{prop-Chow}, for any motive $L$ in $\CH(\kk),$ the group $\Hom(L,M)$ is $N$-torsion. Hence the group $\Hom(L,M\otimes M')$ is $N$-torsion as well. Analogously, $\Hom(L,M\otimes M')$ is $N'$-torsion, which implies the vanishing from~(i) by Yoneda lemma.

Now by Proposition~\ref{prop-Chowdecomp}, we see that the Chow motive $M(S\times S')$ is decomposed into powers of the Lefschetz motive and Lefschetz twists of $M$ and $M'.$ This gives an explicit description of the Chow groups of~$S\times S'.$ Comparing this with the tensor product $CH^{\bullet}(S)\otimes_{\ZZ}CH^{\bullet}(S'),$ we obtain~(ii).
\end{proof}

\begin{remark}{\rm
Actually, Propositions~\ref{prop-Chowdecomp},~\ref{prop-Chow} and Corollary~\ref{corol-Chow} remain true if one replaces $\CC$ by any algebraically closed field of infinite transcendence degree over its prime subfield.}
\end{remark}

\section{Grothendieck group $K_0$ and the proof of Theorem \ref{main}}\label{sec:K}

Let us fix notation and recollect some general facts concerning $K_0$\!-groups (general references are \cite[Exp.0]{BGI} and \cite{Ful}).

As above, let $\kk$ be an arbitrary field and by $X$ denote an irreducible smooth projective variety over $\kk$ of dimension $d.$ By~$K_0(X)$ denote the Grothendieck group of vector bundles on $X$ (equivalently, of coherent sheaves
on $X$). Given a vector bundle $E$ on $X,$ by $[E]$ denote its class in $K_0(X)$ (the same for coherent
sheaves).

There is a canonical decreasing filtration $F^pK_0(X),$ $p\geqslant 0,$ by codimension of support: an
element in $K_0(X)$ belongs to $F^pK_0(X)$ if and only if it can be represented as a linear combination
of elements of type $[\F],$ where $\F$ is a coherent sheaf whose support codimension is at least $p.$

One has a natural homomorphism
\begin{equation}\label{eq:varphi}
\varphi\colon CH^{\bullet}(X)\to \gr^{\bullet}_F K_0(X)
\end{equation}
that sends $Z$ to the class of $[\O_Z]$ for an irreducible codimension $p$ subvariety $Z$ in $X.$ Important facts are that $\varphi$ is surjective and $\varphi_{\,\QQ}$ is an isomorphism (the inverse to $\varphi_{\,\QQ}$ is induced by the Chern character).

The abelian group $K_0(X)$ has a natural commutative ring structure with $[E]\cdot[F]:=[E\otimes F]$ for vector bundles $E$ and $F$ on $X.$ This product respects the filtration $F^{\bullet}$:
$$
F^pK_0(X)\cdot F^qK_0(X)\subset F^{p+q}K_0(X)\,.
$$
This induces a graded ring structure on $\gr^{\bullet}_F K_0(X)$ such that $\varphi$ is a morphism of graded rings.

Further, we have an additive homomorphism $\chi\colon K_0(X)\to\ZZ$
that sends $[E]$ to the Euler characteristic $\chi(X,E)$ (this equals the push-forward map associated to the morphism $X\to \Spec(\kk)$). This leads to the pairing
$$
\langle\cdot,\cdot\rangle\colon K_0(X)\otimes K_0(X) \to \ZZ,\quad [E]\otimes[F]\mapsto \chi(X,E\otimes F)\,.
$$
It follows that $\langle F^pK_0(X),F^{d+1-p}K_0(X)\rangle=0$ and we obtain a pairing between
$\gr^p_FK_0(X)$ and $\gr^{d-p}_FK_0(X),$ which we denote by $\gr\,\langle\cdot,\cdot\rangle.$ The map $\varphi$ sends the intersection pairing between Chow groups to the pairing~$\gr\,\langle\cdot,\cdot\rangle.$

By $m$ denote the external product map
\begin{equation}\label{eq:mult}
m\colon K_0(X)\otimes_{\ZZ} K_0(Y)\to K_0(X\times Y)\,.
\end{equation}
By $a$ denote the map
\begin{equation}\label{eq:act}
a\colon K_0(X\times Y)\lto \Hom_{\ZZ}(K_0(X),K_0(Y))\,,\quad
f\mapsto \{\alpha\mapsto \mathrm{pr}_{2*}(\mathrm{pr}_1^*(\alpha)\cdot f)\}\,,
\end{equation}
where $\mathrm{pr}_i$ are natural projections from $X\times Y.$

\begin{proposition}\label{prop-Ksimple}
Let $S$ and $S'$ be surfaces as in Corollary~\ref{corol-Chow}, in particular, $N$ and $N'$ are coprime and Bloch's conjecture holds for $S$ and $S'.$ Then we have a surjection
$$
m\colon K_0(S)\otimes_{\ZZ} K_0(S')\twoheadrightarrow K_0(S\times S')\,.
$$
\end{proposition}
\begin{proof}
It is enough to show that $\gr_F^{\bullet}(m)$ is surjective. Since the natural map
$$
\gr^{\bullet}_F K_0(S)\otimes_{\ZZ} \gr^{\bullet}_F K_0(S')\to
\gr^{\bullet}_F (K_0(S)\otimes_{\ZZ}K_0(S'))
$$
is surjective and $\varphi$ from equation~\eqref{eq:varphi} is also surjective and commutes with the external product map, we are reduced to show surjectivity of the external product map for Chow groups. This is implied by Corollary~\ref{corol-Chow}(ii).
\end{proof}

\begin{remark}{\rm
Using Proposition~\ref{prop-K} below, one can show that $m$ from Proposition~\ref{prop-Ksimple} is actually an isomorphism.
}
\end{remark}

\begin{proof}[Proof of Theorem \ref{main}]  Now Theorem \ref{main} follows from Proposition \ref{prop-Ksimple} directly.
Indeed, consider the semiorthogonal decompositions
$$
\mathbf{D}^b(\coh S)=\langle \D, \A \rangle \quad\text{and}\quad \mathbf{D}^b(\coh S')=\langle \D', \A' \rangle,
$$
where subcategories $\D$ and $\D'$ are generated by exceptional collections. It is evident
that the map $m\colon K_0(S)\otimes_{\ZZ} K_0(S')\to K_0(S\times S')$ is compatible with semiorthogonal decomposition.
Hence, the surjectivity of $m$ implies a surjectivity of the map $K_0(\A)\otimes_{\ZZ} K_0(\A')\to K_0(\A\prd\A').$
We know that $K_0(\A)\cong\Pic(S)_{\tors}$ and $K_0(\A')\cong\Pic(S)_{\tors}.$ Since the orders of these two groups are coprime we
obtain that $K_0(\A)\otimes_{\ZZ} K_0(\A')=0.$ Thus $K_0(\A\prd\A')=0$ too.
\end{proof}

\section{$K$\!-motives and universal phantoms}\label{sec:Kmotives}

Let us discuss what we call $K\!$-motives (see \cite{Manin} for more details). The category of \mbox{$K\!$-motives} $\KM(\kk)$ over a field $\kk$ is defined similarly to the category of Chow motives $\CH(\kk)$ with Chow groups being replaced by $K_0.$ Namely, objects in $\KM(\kk)$ are pairs $KM(X,\pi),$ where $\pi\in K_0(X\times X)$ is a projector with respect to the composition of \mbox{$K_0$\!-correspondences}. We stress that the last integer~$n$ from Chow motive triples is not present for $K\!$-motives, which corresponds to the absence of a canonical grading on $K_0.$ Morphisms between $K$\!-motives are defined similarly as for Chow motives. In particular, one has the $K\!$-motive $KM(X):=KM(X,[\O_{\Delta_X}])$ and the unit object $\uno$ is $KM(\Spec(\kk)).$ The explicit structure of $K_0(\PP^1\times\PP^1)$ implies the isomorphism $KM(\PP^1)\cong\uno\oplus\uno$ (the Lefschetz motive does not appear in this setting).
There is a symmetric tensor structure on $\KM(\kk)$ defined by the standard formula
$KM(X, \pi)\otimes KM(Y, \rho)=KM(X\times Y, \pi\boxtimes\rho),$ where $\boxtimes$ is induced
 by the external product of vector bundles (see \cite{Manin}).

Since for a smooth variety $K\!$-groups are modules over $K_0,$ all $K$\!-groups are well-defined for $K\!$-motives by the formula (see \cite{Qu})
\begin{equation}\label{eq:Ki}
K_i(KM(X,\pi)):=\pi(K_i(X))\quad i\geqslant 0,\quad\text{where}\quad  \pi(\alpha)= \mathrm{pr}_{2*}(\mathrm{pr}_1^*(\alpha)\cdot \pi).
\end{equation}

Any element $\alpha\in K_0(X)$ defines morphisms
$$
\alpha_*\colon\uno\to KM(X),\quad \alpha^*\colon KM(X)\to\uno\,,
$$
and, given $\beta\in K_0(X),$ one has the equalities
\begin{equation}\label{eq:composK}
\beta^*\circ \alpha_*=\langle\alpha,\beta\rangle\cdot\id_{\uno}\,,\quad \alpha_*\circ\beta^*=\beta\boxtimes \alpha\,.
\end{equation}
We say that a $K$\!-motive is {\sf of unit type} if it is a direct sum of finitely many copies of $\uno.$

For a commutative ring $R,$ one defines similarly the category of $K\!$-motives with $R$-coefficients $\KM(\kk)_R$ based on $K_0$\!-groups with $R$-coefficients $K_0(X)_R:=K_0(X)\otimes_{\ZZ}R.$ We denote \mbox{$K\!$-motives} with $R$-coefficients by $KM(X,\pi)_R.$ For simplicity, we use the same notation $\uno$ for the unit object in $\KM(\kk)_R.$ All what was said above about the category $\KM(\kk)$ remains valid for the category $\KM(\kk)_{R}.$

The following proposition shows us that if the motive of $X$ is of  Lefschetz type then the Grothendieck groups of $X$ and $X\times X$
can be easily described.

\begin{proposition}\label{prop-Kend}
Let $R$ be a commutative ring of characteristic zero, i.e. $R$ contains $\ZZ.$ Let $X$ be an irreducible smooth projective variety over $\kk$ such that the Chow motive
$M(X)_R\in \CH(\kk)_R$ is of Lefschetz type. Then the following is true:
\begin{itemize}
\item[(i)]
the group $K_0(X)_R$ is a free $R$-module of rank $n,$ where $n$ is the rank of $CH^{\bullet}(X)_R$ over $R$;
\item[(ii)]
we have an isomorphism
$
m_R\colon K_0(X)_R\otimes_R K_0(X)_R\stackrel{\sim}\longrightarrow K_0(X\times X)_R\,
$
defined by \eqref{eq:mult};
\item[(iii)]
the pairing $
\langle\cdot,\cdot\rangle_R\colon K_0(X)_R\otimes_R K_0(X)_R\to R
$
is unimodular;

\item[(iv)]
we have an isomorphism
$
a_R\colon K_0(X\times X)_R\stackrel{\sim}\longrightarrow \End_R(\,K_0(X)_R\,)\,
$
defined by \eqref{eq:act}.
\end{itemize}
\end{proposition}
\begin{proof}
First let us show~(i). Since $M(X)_R$ is of Lefschetz type, $CH^{\bullet}(X)_R$ is a free $R$-module of finite rank. Thus the morphism $\varphi_R$ defined by equation~\eqref{eq:varphi} has a torsion free source. Since $\varphi_R$ is also surjective and becomes an isomorphism over $R\otimes_{\ZZ}\QQ,$ we see that $\varphi_R$ is an isomorphism (here we use that~$R$ has characteristic zero). In particular, $\gr^{\bullet}_FK_0(X)_R$ is a free $R$-module of rank~$n,$ whence $K_0(X)_R$ is also a free $R$-module of rank $n,$ because there are no non-trivial extensions between free modules.

As for~(ii) an explicit description of Chow groups for Lefschetz motives implies that the external product map for Chow groups
$$
CH^{\bullet}(X)_R\otimes_R CH^{\bullet}(X)_R\to CH^{\bullet}(X\times X)_R
$$
is an isomorphism. Since $\varphi_R$ is an isomorphism (for $X\times X$ as well, because the Chow motive $M(X\times X)_R\cong M(X)_R\otimes M(X)_R$ is of Lefschetz type) and commutes with the external product maps, we see that $\gr^{\bullet}_F(m_R)$ is an isomorphism. This implies that $m_R$ is also an isomorphism.

Now let us prove~(iii). By Lemma~\ref{lemma-inters}, the intersection pairing on $CH^{\bullet}(X)_R$ is unimodular. Since $\varphi_R$ is an isomorphism, the pairing $\gr\,\langle\cdot,\cdot\rangle_R$ is also unimodular. Choose any splitting $K_0(X)_R\cong \gr^{\bullet}_F K_0(X)_R$ and a basis $e_p$ in $\gr^p_FK_0(X)_R$ over $R$ for each $0\leqslant p\leqslant d$ (here~$e_p$ denotes the collection of elements in a basis). Consider two bases in $K_0(X)_R$: $(e_0,\ldots,e_d)$ and $(e_d,\ldots,e_0).$ Then the pairing $\langle\cdot,\cdot\rangle_R$ between $K_0(X)_R$ and itself is given in these bases by a block lawer-triangular matrix with invertible diagonal blocks. Therefore the pairing $\langle\cdot,\cdot\rangle_R$ is unimodular.

Finally, to show~(iv) note that the composition $a\circ m$ is induced by the morphism \mbox{$K_0(X)\to K_0(X)^{\vee}$} that corresponds to the pairing $\langle\cdot,\cdot\rangle$ (this is true for any~$X$). Since $\langle\cdot,\cdot\rangle_R$ on $K_0(X)_R$ is unimodular, we see that $a_R\circ m_R$ is an isomorphism. As $m_R$ is an isomorphism, we conclude that $a_R$ is an isomorphism as well.
\end{proof}

The following proposition is implied by Theorem~1.7 in~\cite{MT} for the case when $R$ contains~$\QQ.$ It is important for our main result, Theorem~\ref{main2}, that the statement is still true over a more general ring~$R.$ Note that the proof in this more general case requires new arguments, namely, Proposition~\ref{prop-Kend}.

\begin{proposition}\label{prop-K}
Assume that $R$ has characteristic zero. Let $X$ be an irreducible smooth projective variety such that the Chow motive $M(X)_R$ is of Lefschetz type. Then the $K\!$-motive $KM(X)_R$ is of unit type, namely, $KM(X)_R\cong\uno^{\oplus n},$ where $n$ is the rank of $CH^{\bullet}(X)_R$ over~$R.$
\end{proposition}
\begin{proof}
We use all results from Proposition~\ref{prop-Kend}. Let $\{\alpha_i\},$ \mbox{$1\leqslant i\leqslant n$}, be any basis in $K_0(X)_R$ over $R.$ Let $\{\beta_i\},$ $1\leqslant i\leqslant n,$ be the dual basis in $K_0(X)_R$ over $R$ with respect to the unimodular pairing $\langle\cdot,\cdot\rangle_R.$ Put $\pi:=\sum\limits_{i=1}^n\alpha_i\boxtimes\beta_i.$ By equation~\eqref{eq:composK}, $\pi$ is a projector and there is an isomorphism of \mbox{$K$\!-motives} $KM(X,\pi)_R\cong\uno^{\oplus n}.$

Using formula~\eqref{eq:composK} and evaluating $a(\pi)$ at each $\alpha_i,$ we see that $\pi$ acts identically on $K_0(X)_R.$ Since $a_R$ is an isomorphism, we conclude that $\pi=[\O_{\Delta_X}]$ in $K_0(X\times X)_R.$ Consequently, $KM(X,\pi)_R\cong KM(X)_R,$ which finishes the proof.
\end{proof}

The next statement follows immediately from Proposition~\ref{prop-K}.

\begin{corollary}\label{corol-Ksummand}
Assume that the ring $R$ has characteristic zero. Let $X$ be an irreducible smooth projective variety such that the Chow motive $M(X)_R$ is of Lefschetz type. Let $P$ be a direct summand in the $K$\!-motive $KM(X)_R.$ Then $P=0$ in $\KM(\kk)_R$ if and only if $K_0(P)_R=0.$
\end{corollary}

Let $X$ be an arbitrary irreducible smooth projective variety over $\kk.$

To any admissible subcategory $\N\subset\mathbf{D}^b(\coh X)$ we can attach a $K$\!-motive
$KM(\N).$ It is a direct summand of the $K$\!-motive $KM(X).$ Let us explain this in more details.

Denote by $j$ the  functor of inclusion
 of $\N$ to $\mathbf{D}^b(\coh X).$ A splitting of $KM(\N)$ as a direct summand of $KM(X)$ depends on a projection
functor $p\colon \mathbf{D}^b(\coh X)\to\N,$ where the projection~$p$ satisfies $p \circ j\cong\id_{\N}.$
For example, as a projection $p$ we can take a right adjoint functor to the inclusion functor $j.$
The composition $\Phi=j\cdot p$ is a projector functor from $\mathbf{D}^b(\coh X)$ to itself, i.e. $\Phi\circ\Phi\cong\Phi,$
and the image of this functor is exactly the subcategory $\N\subset\mathbf{D}^b(\coh X).$

An enhancement of the category $\mathbf{D}^b(\coh X)$ induces an enhancement of the subcategory $\N$ and the inclusion functor
$j$ is a DG functor between enhancements. As the right adjoint $p$ is also a quasifunctor, i.e. a DG functor up to a quasi-equivalence.
By  To\"en's theorem \cite[Th. 8.15]{To} any  quasifunctor can be represented by an object on the product.
This means that there is an object $\E\in\mathbf{D}^b(\coh(X\times X))$  that is represented the functor $\Phi=j\cdot p,$ i.e. $\Phi\cong\Phi_{\E},$ where
\begin{equation}\label{eq:kernelfunctor}
\Phi_{\E}(-):=\pr_{2*}(\pr_1^*(-)\stackrel{\bf{L}}{\otimes}\E).
\end{equation}
The existence of $\E$ is also shown explicitly in~\cite[Th.~3.7]{Ku1}.
Consider the class $[\E]$ of the object $\E\in\mathbf{D}^b(\coh(X\times X))$
in $K_0(X\times X).$ Since the functor $\Phi_{\E}$ is a projector, the element $[\E]\in K_0(X\times X)$ is a projector
with respect to the composition of $K_0$\!-correspondences and we obtain a $K$\!-motive $KM(X, [\E]).$
It can be easily checked that for different choices of a projection $p$ we obtain isomorphic $K$\!-motives.
Denote it by $KM(\N).$

It is also evident that the $K$\!-groups $K_i(KM(X, [\E]))$ defined for the $K$\!-motive \mbox{$KM(X, [\E])=KM(\N)$} in (\ref{eq:Ki}) coincide with
the $K$\!-groups $K_i(\N)$ because they split off from the groups $K_i(X)$ by the same operator $a([\E])$ acting on $K_i(X)$
by the rule
$
a([E])(-)=\pr_{2*}(\pr_1^*(-)\cdot[\E]).
$

With any two objects $\E\in\mathbf{D}^b(\coh(X\times X))$ and $\F\in\mathbf{D}^b(\coh(Y\times Y))$ we can associate
an object $\E\prd\F\in\mathbf{D}^b(\coh((X\times Y)\times (X\times Y)))$ that is a tensor product of pull backs of
$\E$ and $\F$ on $(X\times Y)\times (X\times Y).$
Assume that $\Phi_{\E}$ and $\Phi_{\F}$ are projectors onto admissible subcategories \mbox{$\N\subset\mathbf{D}^b(\coh X)$} and
$\M\subset\mathbf{D}^b(\coh Y),$ respectively. Then
the functor $\Phi_{\E\prd\F}$ from \mbox{$\mathbf{D}^b(\coh(X\times Y))$}
to itself is also a projector and it projects $\mathbf{D}^b(\coh(X\times Y))$ onto the admissible subcategory \mbox{$\N\prd\M\subset \mathbf{D}^b(\coh(X\times Y))$}.
This implies that
\begin{equation}\label{eq:tens}
KM(\N\prd\M)\cong KM(X\times Y, [\E\prd\F])\cong KM(X, [\E])\otimes KM(Y, [\F])\cong KM(\N)\otimes KM(\M).
\end{equation}

Further, Hochschild homology are well-defined for $K$\!-motives. Let $f\in K_0(X\times Y)$ be a $K_0$\!-correspondence. Let also $\E$ be an object in ${\mathbf D}^b(X\times Y)$ whose class in $K_0(X\times Y)$ equals~$f$. Then a DG-functor $\Phi_{\E}$ defined as in formula~\eqref{eq:kernelfunctor} induces a map $\phi_{\E}\colon HH_i(X)\to HH_i(Y)$. Moreover, $\phi_{\E}$ depends only on the class of $\E$ in $K_0(X\times Y)$, that is, we have a well-defined map $f\colon HH_i(X)\to HH_i(Y)$ (see, for example, \cite{Ke1} page 7, \cite{Ke2} Th.5.2 a), or \cite{Ku1} method of  the proof of Th.7.3). If $\kk$ has characteristic zero, this can be also shown by an explicit formula
$$
 f(\alpha):=\mathrm{pr}_{2*}(\mathrm{pr}_1^*(\alpha)\cdot
\mathrm{pr}_1^*\sqrt{\mathrm{td}_X}\cdot\mathrm{ch}(f)\cdot \mathrm{pr}_2^*\sqrt{\mathrm
{td}_Y}),\quad \alpha\in \mathrm{HH}_i(X)\,,
$$
where $\mathrm{ch}\colon K_0(X\times X)\to \mathrm{HH}_0(X\times X)$ is the Chern character, while $\mathrm{td}_X$  is the Todd class of~$X$ and the analogous for $\mathrm{td}_Y$ (see, for example, \cite{Or}).
Thus Hochschild homology are well-defined for $K$\!-motives by the
formula
$$
\mathrm{HH}_i(KM(X,\pi)):=\pi(\mathrm{HH}_i(X))\,\quad i\geqslant 0\,.
$$

Given an admissible subcategory $\N\subset\mathbf{D}^b(\coh X)$ one checks that there is a canonical isomorphism
$$
\mathrm{HH}_i(KM(\N))\cong \mathrm{HH}_i(\N)\,,
$$
where $KM(\N)$ is the $K$\!-motive associated with $\N$ as above.

Now it can be shown that the property for an admissible subcategory to be a universal phantom is equivalent to the vanishing of its $K$\!-motive.

\begin{proposition}\label{prop:KM-univ}
Let $X$ be a smooth projective variety and let $\N\subset\mathbf{D}^b(\coh X)$ be an admissible subcategory.
Then the following statements are equivalent:
\begin{itemize}
\item[(i)] the subcategory $\N$ is a universal phantom;
\item[(ii)] $K_0(\N\prd\mathbf{D}^b(\coh X))=0;$
\item[(iii)] $KM(\N)=0$ in the category of $K$\!-motives $\KM(\kk).$
\end{itemize}
\end{proposition}
\begin{proof} (i)$\Rightarrow $(ii) This is evident.

(ii)$\Rightarrow $(iii) Since $K_0(\N\prd\mathbf{D}^b(\coh X))=0$
we have
$\Hom_{\KM(\kk)}(KM(X), KM(\N))=0.$
On the other hand, we know that $KM(\N)$ is a direct summand  of $KM(X).$ This immediately implies that $KM(\N)=0.$

(iii)$\Rightarrow$(i) Suppose that $KM(\N)=0$ in the category of $K$\!-motives. Therefore, for any $Y$ we have that
$$
K_0(\N\prd\mathbf{D}^b(\coh Y))=\Hom_{\KM(\kk)}(KM(Y), KM(\N))=0.
$$

Since Hochschild homology are well-defined for $K$\!-motives we obtain that $\N$ is a universal phantom.
\end{proof}

\begin{theorem}\label{th-fin}
Let $X$ and $X'$ be two smooth projective varieties and let $N$ and $N'$ be two coprime integers such that
the motives $M(X)_{1/N}\in\CH(\kk)_{1/N}$ and $M(X')_{1/N'}\in\CH(\kk)_{1/N'}$ are of Lefschetz type.
Let $\A\subset\mathbf{D}^b(\coh X)$ and $\A'\subset\mathbf{D}^b(\coh X')$ be two quasiphantoms such that
$N\cdot K_0(\A)=0$ and $N'\cdot K_0(\A')=0.$ Then the category $\A\prd\A'$ is a universal phantom and
 $KM(\A\prd\A')=0.$
\end{theorem}
\begin{proof}
By Corollary \ref{corol-Ksummand} we have $KM(\A)_{1/N}=0$ and $KM(\A')_{1/N'}=0.$ The \mbox{$K$\!-motive} \mbox{$KM(\A\prd\A')$} is the tensor product of $K$\!-motives
$KM(\A)$ and $KM(\A').$ Therefore \mbox{$KM(\A\prd\A')$} becomes trivial after tensoring with both $\ZZ[1/N]$ and $\ZZ[1/N'].$
Since $N$ and~$N'$ are coprime the $K$\!-motive $KM(\A\prd\A')$ is trivial by itself.
By Proposition \ref{prop:KM-univ} this implies that $\A\prd\A'$ is a universal phantom.
\end{proof}

Now this theorem with Proposition \ref{prop-Chowdecomp} imply our main Theorem \ref{main2} directly.

\begin{proof}[Proof of Theorem \ref{main2}] Denote by $N$ and $N'$ the orders of the groups $K_0(\A)\cong\Pic(S)_{\tors}$
and $K_0(\A')\cong\Pic(S')_{\tors},$ respectively. By assumption, $N$ and $N'$ are coprime.
By Proposition \ref{prop-Chowdecomp} motives $M(S)_{1/N}$ and $M(S')_{1/N'}$ are of Lefschetz type. Thus by Theorem \ref{th-fin}
$\A\prd\A'$ is a universal phantom and $KM(\A\prd\A')=0.$ Since $K$\!-groups of $\A\prd\A'$ coincide with $K$\!-groups of
the $K$\!-motive $KM(\A\prd\A')$ we get a vanishing $K_i(\A\prd\A')=0$ for all $i.$
\end{proof}

\begin{remark}{\rm
It was proved in the paper
\cite{Ta} that any additive invariant on the category of all small DG categories localized with respect to Morita equivalences
factors through a category of noncommutative $K$\!-motives for DG categories, which contains the category
$\KM(\kk)$ as a full subcategory (see also \cite{Ke2}). This means that any natural additive functor
is trivial on a universal phantom category.
}
\end{remark}

\section{$K$\!-motives and Hochschild homology}\label{sec:add}

In our main construction, we have establish the vanishing of Hochschild homology by an explicit use of K\"unneth formula (see~Remark~\ref{rem:Hochsch}). However, over $\CC$ the vanishing of Hochschild homology for admissible subcategories is a consequence of the vanishing of $K_0$\!-groups with rational coefficients as we show in Theorem~\ref{thm-trivial}.

First we give a general result about $K$\!-motives with rational coefficients. The proof is a direct analogue of Lemma~1 from~\cite{GG}, where one considers Chow motives instead of $K\!$-motives.

\begin{proposition}\label{prop-add}
Let $N$ be a rational $K\!$-motive over $\CC,$ i.e., an object in $\KM(\CC)_{\QQ},$ such that $K_0(N)_{\QQ}=0.$ Then $N=0.$
\end{proposition}

The proof of the proposition consists of Lemma~\ref{lemma-add1} and Lemma~\ref{lemma-add2}. In their proofs, $X$ is a smooth projective variety over $\CC$ and $\pi\in K_0(X\times X)_{\QQ}$ is a projector such that $N=KM(X,\pi).$

\begin{lemma}\label{lemma-add1}
Let $N$ be as in Proposition~\ref{prop-add}. Then for any field $F$ over $\CC,$ we have \mbox{$K_0(N_F)_{\QQ}=0$}, where $N_F$ denotes the extension of scalars of $N$ from $\CC$ to $F.$
\end{lemma}
\begin{proof}
Let $\kk$ be a finitely generated field over $\QQ$ such that $X$ and $\pi$ are defined over $\kk,$ i.e., there is an embedding of fields $\kk\subset \CC,$ a smooth projective variety $Y$ over $\kk,$ a projector $\rho\in K_0(Y\times Y)_{\QQ},$ and an isomorphism of varieties over $\CC$
$$
Y\times_{\kk}\CC\cong X
$$
such that $\rho_{\CC}=\pi.$ Let $P:=KM(Y,\rho)$ be the corresponding rational $K\!$-motive over $\kk.$ By construction, we have $P_{\CC}\cong N.$

It is enough to treat the case when $F$ is finitely generated over $\CC.$ Consider an arbitrary element $\alpha\in K_0(N_F)_{\QQ}.$ Let $E$ be a finitely generated field over $\kk$ such that $\alpha$ is defined over~$E,$ i.e., there is an embedding of fields $E\subset F$ over $\kk$ and an element $\beta\in K_0(P_E)_{\QQ}$ such that $\beta_F=\alpha.$

Further, consider any embedding of $E$ in $\CC$ over $\kk.$ We have an embedding of $K_0$-groups with rational coefficients
$$
K_0(P_E)_{\QQ}=\rho_E(K_0(Y_E)_{\QQ})\subset K_0(N)_{\QQ}=\pi(K_0(X)_{\QQ})\,.
$$
Indeed, $K_0$-groups of varieties are embedded for finite extensions by the existence of the push-forward map and are embedded for extensions of algebraically closed fields by taking points on varieties (cf.~\cite[p.1.21]{Bloch}).

By assumption of the lemma, we have $K_0(N)_{\QQ}=0,$ whence $K_0(P_E)_{\QQ}=0,$ $\alpha'=0,$ and $\alpha=0,$ which finishes the proof.
\end{proof}

Let $V$ be a variety over $\CC$ (not necessarily smooth or projective). By $K_0'(V)$ denote the Grothendieck group of coherent sheaves on $V$ (analogously with rational coefficients). Let us define the group $K_0'(N\times V)_{\QQ}.$ For this purpose define the action of $\pi$ on $K'_0(X\times V)_{\QQ}$ by the standard formula
$$
\pi(\alpha):=\pr_{2*}(\pr_1^*(\alpha)\cdot \pi)\,,
$$
where $\pr_i\colon X\times X\times V\to X\times V$ are natural projections. In order to justify this formula note that the morphisms $\pr_i$ are flat and projective and there is a product between $K_0'$-groups and $K_0$-groups (see~\cite{Qu} for more details on functoriality of $K'_0$-groups). We put
$$
K_0'(N\times V)_{\QQ}:=\pi(K'_0(X\times V)_{\QQ})\,.
$$
Note that this agrees with the case when $V$ is smooth and projective.

\begin{lemma}\label{lemma-add2}
Let $N$ be a rational $K\!$-motive over $\CC$ such that for any field $F$ over $\CC,$ we have $K_0(N_F)_{\QQ}=0.$  Then the following is true:
\begin{itemize}
\item[(i)]
for any variety $V$ over $\CC,$ we have
$
K_0'(N\times V)_{\QQ}=0\,;
$
\item[(ii)]
we have $N=0.$
\end{itemize}
\end{lemma}
\begin{proof}
First let us prove (i) by induction on dimension of $V.$ We have an exact sequence
$$
\bigoplus\limits_W K'_0(X\times W)\lto K_0'(X\times V)\lto \bigoplus\limits_i{K_0'(X_{F_i})}\lto 0\,,
$$
where $W$ runs through all subvarieties in $V$ of smaller dimension and $F_i$ runs through fields of rational functions on irreducible components in $V$ of maximal dimension. Taking this exact sequence with rational coefficients and applying $\pi,$ we obtain an exact sequence
$$
\bigoplus\limits_W K'_0(W\times N)_{\QQ}\lto K_0'(V\times N)_{\QQ}\lto \bigoplus\limits_i{K_0'(N_{F_i})_{\QQ}}\lto 0\,.
$$
The condition of the lemma immediately implies the required statement.
Finally, (ii) follows from (i) by Yoneda lemma when we assume $V$ to be smooth and projective.
\end{proof}

\begin{remark}{\rm
It follows from the proof of Proposition~\ref{prop-add} that we may replace in its formulation the field of definition $\CC$ by any algebraically closed field of infinite transcendence degree over its prime subfield.
}
\end{remark}

Since higher $K\!$-groups and Hochschild homology are well-defined for $K\!$-motives, we obtain the following result by Proposition~\ref{prop-add}.

\begin{theorem}\label{thm-trivial}
Let $X$ be a smooth projective variety over $\CC,$ $\N\subset \mathbf{D}^b(\coh X)$ be an admissible subcategory.  Suppose that $K_0(\N)_{\QQ}=0.$ Then $K_i(\N)_{\QQ}=0$ and $\mathrm{HH}_i(\N)=0$ for all $i\geqslant 0.$
\end{theorem}

\end{document}